\documentclass[final]{amsart}

\usepackage{graphicx}
\usepackage{amsmath,amsthm,amsfonts,amscd,amssymb}
\usepackage{url}
\usepackage[color,notref,notcite]{showkeys} 
\definecolor{labelkey}{rgb}{0.6,0,1}

\theoremstyle{plain}
\newtheorem{theorem}{Theorem}[section]

\newtheorem{lemma}[theorem]{Lemma}

\newtheorem{proposition}[theorem]{Proposition}

\newtheorem{corollary}[theorem]{Corollary}

\theoremstyle{definition}
\newtheorem{definition}[theorem]{Definition}

\def\bhyp#1{\begin{equation}\label{#1}\begin{array}{c}}
\def\ehyp{\end{array}\end{equation}}

\newcounter{cst}
\def \ctel#1{C_{\refstepcounter{cst}\label{#1}\thecst}}
\def \cter#1{C_{\ref{#1}}}

\theoremstyle{remark}
\newtheorem{remark}{Remark}[section]

\numberwithin{equation}{section}
\numberwithin{figure}{section}

\newcommand{\cM}{{\mathcal M}}

\newcommand{\RR}{{\mathbb R}}
\newcommand{\NN}{{\mathbb N}}

\def\O{\Omega}
\def\dsp{\displaystyle}
\usepackage[normalem]{ulem}
\definecolor{violet}{rgb}{0.580,0.,0.827}

\DeclareMathOperator{\divergence}{\mathrm{div}} 
\newcommand{\ud}{\, \mathrm{d}} 
\def\div{\divergence}

\makeatletter
\newcommand*\if@single[3]{%
  \setbox0\hbox{${\mathaccent"0362{#1}}^H$}%
  \setbox2\hbox{${\mathaccent"0362{\kern0pt#1}}^H$}%
  \ifdim\ht0=\ht2 #3\else #2\fi
  }
\newcommand*\rel@kern[1]{\kern#1\dimexpr\macc@kerna}
\newcommand*\widebar[1]{\@ifnextchar^{{\wide@bar{#1}{0}}}{\wide@bar{#1}{1}}}
\newcommand*\wide@bar[2]{\if@single{#1}{\wide@bar@{#1}{#2}{1}}{\wide@bar@{#1}{#2}{2}}}
\newcommand*\wide@bar@[3]{%
  \begingroup
  \def\mathaccent##1##2{%
    \if#32 \let\macc@nucleus\first@char \fi
    \setbox\z@\hbox{$\macc@style{\macc@nucleus}_{}$}%
    \setbox\tw@\hbox{$\macc@style{\macc@nucleus}{}_{}$}%
    \dimen@\wd\tw@
    \advance\dimen@-\wd\z@
    \divide\dimen@ 3
    \@tempdima\wd\tw@
    \advance\@tempdima-\scriptspace
    \divide\@tempdima 10
    \advance\dimen@-\@tempdima
    \ifdim\dimen@>\z@ \dimen@0pt\fi
    \rel@kern{0.6}\kern-\dimen@
    \if#31
      \overline{\rel@kern{-0.6}\kern\dimen@\macc@nucleus\rel@kern{0.4}\kern\dimen@}%
      \advance\dimen@0.4\dimexpr\macc@kerna
      \let\final@kern#2%
      \ifdim\dimen@<\z@ \let\final@kern1\fi
      \if\final@kern1 \kern-\dimen@\fi
    \else
      \overline{\rel@kern{-0.6}\kern\dimen@#1}%
    \fi
  }%
  \macc@depth\@ne
  \let\math@bgroup\@empty \let\math@egroup\macc@set@skewchar
  \mathsurround\z@ \frozen@everymath{\mathgroup\macc@group\relax}%
  \macc@set@skewchar\relax
  \let\mathaccentV\macc@nested@a
  \if#31
    \macc@nested@a\relax111{#1}%
  \else
    \def\gobble@till@marker##1\endmarker{}%
    \futurelet\first@char\gobble@till@marker#1\endmarker
    \ifcat\noexpand\first@char A\else
      \def\first@char{}%
    \fi
    \macc@nested@a\relax111{\first@char}%
  \fi
  \endgroup
}
\makeatother

\title[Porous media flow with measure data]{On a miscible displacement model in porous media flow with measure data}

\author{J\'er\^ome Droniou}
\address[J\'er\^ome Droniou]{School of Mathematical Sciences, Monash University, Victoria 3800, Australia.}
\email{jerome.droniou@monash.edu}

\author{Kyle S. Talbot}
\address[Kyle S. Talbot]{School of Mathematical Sciences, Monash University, Victoria 3800, Australia.}
\email{kyle.talbot@monash.edu}



\date{\today}

\begin{document}


\begin{abstract}
We establish the existence of a solution to a non-linearly coupled elliptic-parabolic system of PDEs describing the single-phase, miscible displacement of one incompressible fluid by another in a porous medium. We consider a velocity-dependent diffusion-dispersion tensor and model the action of injection and production wells by measures on the domain. We obtain the solution by passing to the limit on problems with regularised well data.  
\end{abstract}

\maketitle


\section{Introduction} \label{sec:intro}

The single-phase, miscible displacement of one incompressible fluid by another in a porous medium occurs during the enhanced oil recovery process in petroleum engineering. Engineers displace the oil in a reservoir by the injection of a fluid, such as a polymeric solvent, into designated injection wells. A mixture of the oil and the invading fluid is then collected at production wells. Under the appropriate physical hypotheses, such a miscible displacement problem is described by a non-linearly coupled elliptic-parabolic system, occasionally referred to as the Peaceman model after it was first introduced by Peaceman and Rachford in \cite{pr62}. We refer the reader to \cite{bea72,cj86,pe77} for further details.

The reservoir in the physical problem is represented by a bounded domain $\Omega \subset \mathbb{R}^d$ ($d=2$ or $3$) with Lipschitz continuous boundary $\partial \Omega$, and we denote by $(0, T)$ the time interval over which the process occurs. 

In that which follows, $\Phi$ denotes the porosity of the reservoir and $\mathbf{K}$ the absolute permeability of the reservoir. The Darcy velocity of the fluid mixture is denoted by $\mathbf{u}$, and the concentration-dependent quantities of viscosity and density by $\mu$ and $\rho$, respectively. We write $\mathbf{g}$ for the constant, downward-pointing gravitational vector. The main unknowns of the problem are the pressure $p$ of the fluid mixture and the concentration $c$ of one of the components in the mixture.  With this notation in place, the model is given by
\begin{equation}
\left.
\begin{array}{r}
\displaystyle{\mathbf{u}(x, t) = -\frac{\mathbf{K}(x)}{\mu(c(x, t))}\big(\nabla p(x, t) - \rho(c(x, t))\mathbf{g}\big)} \\[0.5cm]
\displaystyle{\divergence\mathbf{u}(x, t) = (q^{I} - q^{P})(x, t)}
\end{array}
\right\}, \mbox{\quad $(x, t) \in \Omega \times (0, T)$,} \label{eq:darcy}
\end{equation}
\begin{multline}
\Phi(x)\partial_{t}c(x, t) - \divergence\big(\mathbf{D}(x, \mathbf{u}) \nabla c - c\mathbf{u}\big)(x, t) + (q^{P}c)(x, t) = (\hat{c}q^{I})(x, t),\\
\mbox{\quad $(x, t) \in\Omega \times (0, T)$.}\label{eq:conc}
\end{multline}
The sums of injection well source terms and production well sink terms are given by $q^{I}$ and $q^{P}$, respectively, and the concentration of the injected fluid is $\hat c$. After Peaceman \cite{pe66}, the diffusion-dispersion tensor $\mathbf{D}$ frequently assumes the form
\begin{equation*}
\mathbf{D}(x, \mathbf{u}) = \Phi(x)\bigg(d_{m}\mathbf{I} + |\mathbf{u}|\Big(d_{l}E(\mathbf{u}) + d_{t}(\mathbf{I} - E(\mathbf{u}))\Big)\bigg), 
\end{equation*}
where $\mathbf{I}$ is the identity matrix, $d_m > 0$ the molecular diffusion coefficient and $d_l > 0$ and $d_t > 0$ the longitudinal and transverse mechanical dispersion coefficients, respectively. In practice, $d_{m}$ is very small relative to the mechanical dispersion coefficients, and $d_l$ is much larger than $d_t$ \cite{dew83}. The matrix $E(\mathbf{u})$ is the projection along the direction of flow:
\begin{equation*}
E(\mathbf{u}) = \Big(\frac{\mathbf{u}_i \mathbf{u}_j}{|\mathbf{u}|^{2}} \Big)_{1 \leq i,j \leq d}. 
\end{equation*}

Following \cite{kov63}, the viscosity is usually given by
\begin{equation*}
\mu(c) = \mu(0)\Big(1 + \big(M^{1/4} - 1\big)c\Big)^{-4} \mbox{\quad for $c \in [0, 1]$,} 
\end{equation*}
where $M = \mu(0)/\mu(1)$ is the mobility ratio. Note that if $M=1$ the system \eqref{eq:darcy}--\eqref{eq:conc} is decoupled.

We assume the reservoir boundary $\partial \Omega$ to be impermeable, which yields no-flow boundary conditions for \eqref{eq:darcy}--\eqref{eq:conc}:
\begin{align}
\mathbf{u}(x, t) \cdot \mathbf{n} &= 0, \mbox{\quad $(x, t) \in \partial \Omega \times (0, T)$, and} \label{eq:bc1}\\
\mathbf{D}(x, \mathbf{u})\nabla c(x, t) \cdot \mathbf{n} &= 0, \mbox{\quad $(x, t) \in \partial\Omega \times (0, T)$,} \label{eq:bc2}
\end{align}
where $\mathbf{n}$ is the exterior unit normal to $\partial \Omega$. The first of these enforces a compatibility condition upon the source terms:
\begin{equation*}
\int_{\Omega} q^{I}(x, t) \ud x = \int_{\Omega} q^{P}(x, t) \ud x \mbox{\quad for all $t \in (0, T)$,} 
\end{equation*}
and we normalise the pressure to eliminate any arbitrary constants:
\begin{equation}
\int_{\Omega} p(x, t) \ud x = 0 \mbox{\quad for all $t \in (0, T)$.} \label{eq:normp}
\end{equation}
Additionally, we specify an initial condition for the concentration:
\begin{equation}
c(x,0) = c_{0}(x), \mbox{ \quad $x \in \Omega$.} \label{eq:ic}
\end{equation}

The source terms $q^{I}$ and sink terms $q^{P}$, henceforth collectively referred to as `source terms', give the flow rates of the injected and produced fluids at the respective wells. The principle difficulty in their mathematical representation is that of scale. Typical wellbore diameters are fractions of a metre, whereas the diameter of a reservoir may be up to thousands of metres. At the former scale, a source term may be adequately modelled as a classical function supported on the wellbore. However, at the reservoir scale -- that is, on the order of a typical mesh size of a discretisation used in reservoir simulation -- the action of a well is effectively that of a (spatial) measure supported at a point in two-dimensional models, or on a line in three-dimensional models. 

Studies addressing the well-posedness of the Peaceman model are few relative to numerical treatments of the problem; see Feng \cite{fe02} for a survey of existence results. We note that Sammon \cite{sam86} models the wells as mollified Dirac masses and assumes the diffusion coefficient $\mathbf{D}$ to be independent of velocity. Subsequent work by Mikeli\'c \cite{mik91} on the stationary problem with a slightly regularised diffusion-dispersion tensor models the wells as non-negative elements of $L^{r}(\Omega)$, $r \in (d, \infty)$. Both Feng \cite{fe95} and Chen and Ewing \cite{ce99} consider the diffusion-dispersion tensor as given above; the former takes $q^{I}, q^{P} \in L^{\infty}(0,T; L^{2}(\Omega))$, whereas the latter relaxes the regularity of the production well to $q^{P} \in L^{\infty}(0,T; H^{-1}(\Omega))$ and establishes existence for a variety of boundary conditions. Amirat and Ziani \cite{az04} assume $L^{\infty}(0,T; L^{2}(\Omega))$ regularity of the wells and prove the existence of a weak solution when $d_{m}=0$, a condition motivated by the negligible contribution at high flow velocities of molecular diffusion relative to mechanical dispersion. 

Under the assumption that $\mathbf{D}$ is bounded, Fabrie and Gallou\"et \cite{fg00} provide the first existence result for miscible displacement with wells modelled by measures. However, this assumption on $\mathbf{D}$ is incompatible with the form of Peaceman's tensor above since the Darcy velocity $\mathbf{u}$ is potentially unbounded. Choquet \cite{ch04} considers this issue in the context of a compressible model of radionuclide transport with wells modelled by measures. The coupling of the equations in such compressible models is stronger than in the incompressible model we study in this paper, however it also provides estimates on the time derivative of the pressure which entail straightforward compactness results on this unknown. For the incompressible model \eqref{eq:darcy}--\eqref{eq:conc}, additional arguments are required to establish a convergence result on the pressure that permits passage to the limit in the coupled system.

In this paper we extend the work of \cite{fg00} to establish the existence of a weak solution to \eqref{eq:darcy}--\eqref{eq:ic} with a diffusion-dispersion tensor that generalises that of Peaceman and with wells modelled by measures. Section \ref{sec:assump} lists the assumptions on the data and details the main result. In Section \ref{sec:regprob} we show that a weak solution exists for an approximate problem with regularised wells, which gives an alternative proof of a similar result of \cite{ce99}. We prove our main result in Section \ref{sec:measprob} by passing to the limit on the approximate problem using some estimates and a stability result developed therein, the latter of which in particular allows us to dispense with an additional regularity assumption on the source terms required by \cite{fg00}. The principal novelty of our work is perhaps in the treatment of the diffusion-dispersion term when passing to the limit, which is accomplished with reference to some technical lemmas presented in Appendix \ref{sec:app}.


\section{Assumptions and main result} \label{sec:assump}

The assumptions on the data are as follows.
\bhyp{hyp:domain}
T \in \RR_{+}^{\ast} \mbox{ and }\Omega \mbox{ is a bounded, open subset of }\RR^{d}, d\geq1, \\
\mbox{ with a Lipschitz continuous boundary.}
\ehyp
Denoting by $\mathcal{S}_d(\RR)$ the set of $d\times d$ symmetric matrices, the permeability
satisfies
\bhyp{hyp:K}
\mathbf{K}: \Omega \to \mathcal{S}_{d}(\RR) \mbox{ is measurable and}\\
\exists k_{\ast}>0 \mbox{ such that, for a.e. }x \in \Omega \mbox{ and }\forall \xi \in \RR^{d},\\
k_{\ast}|\xi|^{2} \leq \mathbf{K}(x)\xi \cdot \xi \leq k_{\ast}^{-1}|\xi|^{2}.
\ehyp
The assumptions on the porosity, density and viscosity are quite general:
\bhyp{hyp:porosity}
\Phi \in L^{\infty}(\Omega) \mbox{ and there exists } \phi_{\ast} > 0 \mbox{ such that for a.e. $x \in \Omega$, } \\
\phi_{\ast} \leq \Phi(x) \leq \phi_{\ast}^{-1},
\ehyp
\bhyp{hyp:densvisc}
\rho \in C([0,1], \mathbb{R}),\, \mu \in C([0,1],(0,\infty)). 
\ehyp
The diffusion-dispersion tensor satisfies
\bhyp{hyp:D}
\mathbf{D}: \Omega \times \RR^{d} \to \mathcal{S}_{d}(\RR) \mbox{ is a Carath\'eodory function such that}\\
\exists \alpha_{\mathbf{D}} > 0, \, \exists \Lambda_{\mathbf{D}} > 0 \mbox{ such that, for a.e. $x \in \Omega$ and all }\zeta, \xi \in \RR^{d},\\
\mathbf{D}(x, \zeta)\xi \cdot \xi \geq \alpha_{\mathbf{D}}(1 + |\zeta|)|\xi|^{2} \mbox{ and } |\mathbf{D}(x,\zeta)|\leq \Lambda_{\mathbf{D}}(1 + |\zeta|).
\ehyp
Finally, the injected and initial concentration are such that
\bhyp{hyp:injectedconc}
\hat c \in L^{\infty}(0,T; C(\widebar \Omega)) \mbox{ satisfies } 0 \leq \hat c(x, t) \leq 1 \mbox{ for a.e. }(x, t) \in \Omega \times (0, T),
\ehyp
\bhyp{hyp:initialconc}
c_{0} \in L^{\infty}(\Omega) \mbox{ satisfies } 0 \leq c_{0}(x) \leq 1 \mbox{ for a.e. } x \in \Omega,
\ehyp
and the source terms are such that
\bhyp{hyp:ab}
q^{I} = a\nu \mbox{ and }q^{P} = b\nu, \mbox{ where }\\
\nu\in \cM_+(\Omega)\mbox{ and }a, b \in L^{\infty}(0,T; C(\widebar\Omega)) \mbox{ satisfy:}\\
a(x, t) \geq 0, \quad b(x, t) \geq 0 \quad \forall (x, t) \in \Omega \times (0, T),
\ehyp
where $\cM_{+}(\Omega)$ is the set of bounded non-negative Radon measures on $\Omega$.
The compatibility condition imposed by \eqref{eq:bc1} becomes
\begin{equation}
\int_{\Omega}a(x, t)\ud \nu(x) = \int_{\Omega}b(x, t)\ud \nu(x) \quad \forall t \in (0, T). \label{hyp:compatibility}
\end{equation}
For a topological vector space $X(\Omega)$ of functions on $\Omega$, we write $(X(\Omega))'$ for its topological dual and omit $\Omega$ when writing the duality pairing: $\langle \cdot , \cdot \rangle_{(X(\Omega))', X(\Omega)} = \langle \cdot , \cdot \rangle_{X', X}$. For $z \in [1, \infty)$, we write $z'=\frac{z}{z-1}$ for its conjugate exponent. When a constant appears in an estimate we track only its relevant dependencies, which are always non-decreasing. 
\begin{remark}
The particular forms of the viscosity and the diffusion-dispersion tensor given in Section \ref{sec:intro} satisfy the assumptions \eqref{hyp:densvisc}, \eqref{hyp:D}; for the latter, set $\alpha_{\mathbf{D}} = \phi_{\ast}\inf(d_{m}, d_{l}, d_{t})$ and $\Lambda_{\mathbf{D}} = \phi_{\ast}^{-1}\sup(d_{m}, d_{l}, d_{t})$.
\end{remark}

Under these assumptions, we will consider the following notion of solutions to the Peaceman model.


\begin{definition} \label{def:weaksolmes}
Under the hypotheses \eqref{hyp:domain}--\eqref{hyp:compatibility}, a weak solution of \eqref{eq:darcy}--\eqref{eq:ic} is a triple $(p, \mathbf{u}, c)$ satisfying 
\begin{equation}
\begin{array}{c}
c \in L^{2}(0,T; H^{1}(\Omega)), \quad 0 \leq c \leq 1, \\
c \in L^{\infty}(0,T; L^{1}(\Omega, \nu)), \\
0 \leq c(x,t) \leq 1, \mbox{\quad for $\nu$-a.e. $x \in \Omega$, for a.e. $t \in (0,T)$,}
\end{array} \label{eq:measc}
\end{equation}
\begin{equation}
\Phi \partial_{t}c \in  L^{2}(0,T; (W^{1, s}(\Omega))'), \quad \forall s > 2d, \label{eq:measdtc}
\end{equation}
\begin{equation}
\Phi c \in C([0, T];(W^{1,s}(\Omega))'), \quad \Phi c(\cdot, 0) = \Phi c_{0} \mbox{\quad in $(W^{1, s}(\Omega))'$}, \quad \forall s > 2d,\label{eq:measic}
\end{equation}
\begin{equation}
\mathbf{D}(\cdot, \mathbf{u})\nabla c \in L^{2}(0,T; L^{r}(\Omega)^{d}), \quad \forall r < \frac{2d}{2d-1}, \label{eq:measdgradc}
\end{equation}
\begin{equation}
p \in L^{\infty}(0,T; W^{1, q}(\Omega)), \quad \mathbf{u} \in L^{\infty}(0,T; L^{q}(\Omega)^d), \quad \forall q < \frac{d}{d-1}, \label{eq:measpu}
\end{equation}
\begin{equation}
\begin{array}{l}
\dsp \int_{0}^{T}\langle \Phi \partial_{t}c(\cdot, t), \varphi(\cdot, t)\rangle_{(W^{1, s})', W^{1, s}}\ud t \\
\dsp + \int_{0}^{T}\int_{\Omega} \mathbf{D}(x, \mathbf{u}(x, t))\nabla c(x, t) \cdot \nabla \varphi(x, t)\ud x \ud t  \\
\dsp - \int_{0}^{T}\int_{\Omega} \mathbf{u}(x, t)\cdot \nabla \varphi(x, t)c(x, t) \ud x \ud t \\
\dsp + \int_{0}^{T}\int_{\Omega} c(x, t)\varphi(x, t)b(x, t)\ud \nu(x) \ud t\\
\dsp =\int_{0}^{T}\int_{\Omega} \hat c(x, t)\varphi(x, t)a(x, t)\ud \nu(x) \ud t,
  \ \forall \varphi \in L^{2}(0,T; W^{1, s}(\Omega)), \forall s>2d,
\end{array}
\label{eq:measpara}
\end{equation}
\begin{equation}
\begin{array}{l}
\dsp \mathbf{u}(x, t) =  -\frac{\mathbf{K}(x)}{\mu(c(x, t))}(\nabla p(x,t) - \rho(c(x, t))\mathbf{g}),  \\
\dsp - \int_{0}^{T}\int_{\Omega}\mathbf{u}(x, t) \cdot \nabla \psi(x, t) \ud x \ud t = \int_{0}^{T}\int_{\Omega}\psi(x, t) a(x, t)\ud \nu(x) \ud t  \\
\dsp \qquad - \int_{0}^{T}\int_{\Omega}\psi(x, t)b(x, t)\ud \nu(x) \ud t, \quad \forall \psi \in \bigcap_{q>d} L^{1}(0,T; W^{1,q}(\Omega)).
\end{array}\label{eq:measell} 
\end{equation}
\end{definition}

The main result of this article is the following existence result.
\begin{theorem} \label{th:main}
Under the hypotheses \eqref{hyp:domain}--\eqref{hyp:compatibility}, there exists a weak solution to \eqref{eq:darcy}--\eqref{eq:ic} in the sense of Definition \ref{def:weaksolmes}. 
\end{theorem}

\begin{remark}
The proof actually shows that $(p,\mathbf{u})$ is a solution to \eqref{eq:darcy}--\eqref{eq:bc1} in a stronger sense than what we state in 
\eqref{eq:measell}; see the proof of Proposition \ref{prop:stab}.
\end{remark}

We prove Theorem \ref{th:main} by passing to the limit on problems with a regularised $\nu$ as defined in Proposition \ref{prop:dreg} below. Whilst the result of this proposition is already established in \cite{ce99}, we give an alternative proof based on the regularisation method we then use to prove Theorem \ref{th:main}. This serves to present the estimates we require and illustrate the techniques we use for the singular problem, in particular the technical lemmas of Appendix \ref{sec:app}.
 
\begin{proposition} \label{prop:dreg}
Let \eqref{hyp:domain}--\eqref{hyp:compatibility} hold and suppose that $\nu \in L^{2}(\Omega)$. Then there exists a solution $(\widetilde{p}, \widetilde{\mathbf{u}}, \widetilde{c})$ to \eqref{eq:darcy}--\eqref{eq:ic} in the following sense:
\begin{equation}
\widetilde{c} \in L^{2}(0,T; H^{1}(\Omega)), \quad 0 \leq \widetilde{c} \leq 1, \label{eq:regc}
\end{equation}
\begin{equation}
\Phi \partial_{t}\widetilde{c} \in L^{2}(0,T; (W^{1, 4}(\Omega))'), \label{eq:regdtc}
\end{equation}
\begin{equation}
\Phi \widetilde{c} \in C([0, T]; (W^{1, 4}(\Omega))'), \quad \Phi\widetilde{c}(\cdot, 0) = \Phi c_{0} \mbox{\quad in $(W^{1, 4}(\Omega))'$}, \label{eq:regic}
\end{equation}
\begin{equation}
\mathbf{D}(\cdot, \widetilde{\mathbf{u}})\nabla \widetilde{c} \in L^{2}(0,T; L^{4/3}(\Omega)^{d}), \label{eq:dgradc}
\end{equation}
\begin{equation}
\widetilde{p} \in L^{\infty}(0,T; H^{1}(\Omega)), \quad \widetilde{\mathbf{u}} \in L^{\infty}(0,T; L^{2}(\Omega)^d), \label{eq:regpu}
\end{equation}
\begin{equation}
\begin{array}{l}
\dsp \int_{0}^{T}\langle \Phi \partial_{t}\widetilde{c}(\cdot, t), \varphi(\cdot, t)\rangle_{(W^{1, 4})', W^{1, 4}}\ud t \\
\dsp + \int_{0}^{T}\int_{\Omega} \mathbf{D}(x, \widetilde{\mathbf{u}}(x, t))\nabla \widetilde{c}(x, t) \cdot \nabla \varphi(x, t)\ud x \ud t \\
\dsp - \int_{0}^{T}\int_{\Omega} \widetilde{\mathbf{u}}(x, t)\cdot \nabla \varphi(x, t)\widetilde{c}(x, t) \ud x \ud t \\
\dsp + \int_{0}^{T}\int_{\Omega} \widetilde{c}(x, t)\varphi(x, t)b(x, t)\nu(x)\ud x \ud t \\
\dsp =\int_{0}^{T}\int_{\Omega} \hat c(x, t)\varphi(x, t)a(x, t)\nu(x) \ud x \ud t , \quad \forall \varphi \in L^{2}(0,T; W^{1, 4}(\Omega)), 
\end{array}
\label{eq:regpara}
\end{equation}
\begin{equation}
\begin{array}{l}
\dsp \widetilde{\mathbf{u}}(x, t) =  -\frac{\mathbf{K}(x)}{\mu(\widetilde{c}(x, t))}(\nabla \widetilde{p}(x,t) - \rho(\widetilde{c}(x, t))\mathbf{g}),\\
\dsp - \int_{0}^{T}\int_{\Omega}\widetilde{\mathbf{u}}(x, t) \cdot \nabla \psi(x, t) \ud x \ud t = \int_{0}^{T}\int_{\Omega}\psi(x, t) a(x, t)\nu(x) \ud x \ud t  \\
\qquad \dsp - \int_{0}^{T}\int_{\Omega}\psi(x, t)b(x, t)\nu(x) \ud x \ud t, \quad \forall \psi \in L^{1}(0,T; H^{1}(\Omega)) . 
\end{array} 
\label{eq:regell}
\end{equation}
\end{proposition}

\begin{remark} \label{rem:rem1}
It will be seen that in fact $\mathbf{D}^{1/2}(\cdot, \widetilde{\mathbf{u}})\nabla \widetilde{c} \in L^{2}(0,T; L^{2}(\Omega)^{d})$, which, together with \eqref{eq:regpu}, is stronger than \eqref{eq:dgradc}. 
\end{remark}
One obtains the existence of a solution to this regularised problem by passing to the limit on problems with bounded diffusion-dispersion tensors, detailed in the next section.


\section{Existence of a solution to the regularised problem} \label{sec:regprob}
\subsection{Approximating the diffusion-dispersion tensor}
For $k \in \RR_{+}$, consider the truncated diffusion-dispersion tensor
\begin{equation}
\mathbf{D}_{k}(x, \mathbf{u}) := \mathbf{D}\bigg(x, \frac{T_{k}(|\mathbf{u}|)}{|\mathbf{u}|}\mathbf{u}\bigg), \label{eq:dk}
\end{equation}
where $T_{k}: [0, \infty) \to [0,\infty)$ is the truncator
\begin{equation*}
T_{k}(s) = \left\{
\begin{array}{rl}
s & \text{if } s< k\\
k & \text{if } s\geq k.
\end{array} \right.
\end{equation*}
It is simple to verify that for each $k \in \RR_{+}$, $\mathbf{D}_{k}$ is a bounded Carath\'eodory function that satisfies
\begin{equation}
\mathbf{D}_{k}(x, \mathbf{u})\xi \cdot \xi \geq \alpha_{\mathbf{D}}|\xi|^{2}, \mbox{\quad for a.e. $x \in \Omega$, $\forall \xi \in \RR^d$}. \label{eq:dkassump}
\end{equation}
Furthermore, $\mathbf{D}_{k}$ can be estimated uniformly with respect to $k$: 
\begin{align}
|\mathbf{D}_{k}(x, \mathbf{u})| &\leq \Lambda_{\mathbf{D}}(1 + |\mathbf{u}|), \mbox{\quad for a.e. $x \in \Omega$},
\end{align}
and for a.e. $x \in \Omega$, $\mathbf{D}_{k}(x,\cdot) \to \mathbf{D}(x,\cdot)$ uniformly (as $k \to \infty$) on compact sets. 

Now, in \eqref{eq:darcy}--\eqref{eq:bc2} replace $\mathbf{D}$ by $\mathbf{D}_{k}$ and assume that \eqref{hyp:domain}--\eqref{hyp:compatibility} hold ($\mathbf{D}_{k}$ satisfies \eqref{hyp:D}). Assume also that $\nu \in L^{2}(\Omega)$ and $\Phi \equiv 1$ on $\Omega$. Then \cite{fg00} establishes the existence of a solution $(p_{k}, \mathbf{u}_{k}, c_{k})$  to \eqref{eq:darcy}--\eqref{eq:ic} in the following sense:
\begin{equation}
c_{k} \in L^{2}(0,T; H^{1}(\Omega))\cap C([0, T]; L^{2}(\Omega)), \quad 0 \leq c_{k} \leq 1, \label{eq:fgconc}
\end{equation}
\begin{equation}
\partial_{t}c_{k} \in L^{2}(0,T; (H^{1}(\Omega))'),
\end{equation}
\begin{equation}
c_{k}(\cdot, 0) = c_{0} \mbox{\quad in $L^{2}(\Omega)$},
\end{equation}
\begin{equation}
p_{k} \in L^{\infty}(0,T; H^{1}(\Omega)), \quad \mathbf{u}_{k} \in L^{\infty}(0,T; L^{2}(\Omega)^d),
\end{equation}
\begin{equation}
\begin{array}{l}
\dsp \int_{0}^{T}\langle \partial_{t}c_{k}(\cdot, t), \varphi(\cdot, t)\rangle_{(H^{1})', H^{1}}\ud t \\
\dsp + \int_{0}^{T}\int_{\Omega} \mathbf{D}_{k}(x, \mathbf{u}_{k}(x, t))\nabla c_{k}(x, t) \cdot \nabla \varphi(x, t)\ud x \ud t \\
\dsp - \int_{0}^{T}\int_{\Omega} \mathbf{u}_{k}(x, t)\cdot \nabla \varphi(x, t)c_{k}(x, t) \ud x \ud t\\
\dsp + \int_{0}^{T}\int_{\Omega} c_{k}(x, t)\varphi(x, t)b(x, t)\nu(x)\ud x \ud t  \\
\dsp =\int_{0}^{T}\int_{\Omega} \hat c(x, t)\varphi(x, t)a(x, t)\nu(x) \ud x \ud t, \quad \forall \varphi \in L^{2}(0,T; H^{1}(\Omega)),
\end{array}
\label{eq:fgpara}
\end{equation}
\begin{equation}
\begin{array}{l}
\dsp \mathbf{u}_{k}(x, t) =  -\frac{\mathbf{K}(x)}{\mu(c_{k}(x, t))}(\nabla p_{k}(x, t) - \rho(c_{k}(x, t))\mathbf{g}), \\
\dsp - \int_{0}^{T}\int_{\Omega}\mathbf{u}_{k}(x, t) \cdot \nabla \psi(x, t) \ud x \ud t= \int_{0}^{T}\int_{\Omega}\psi(x, t) a(x, t)\nu(x) \ud x \ud t \\
\dsp \qquad - \int_{0}^{T}\int_{\Omega}\psi(x, t)b(x, t)\nu(x) \ud x \ud t, \quad \forall \psi \in L^{1}(0,T; H^{1}(\Omega)).
\end{array}
\label{eq:fgell}
\end{equation}

\begin{remark} \label{rem:rem2}
The assumption of constant unit porosity made in \cite{fg00} can be relaxed to \eqref{hyp:porosity} without the introduction of further mathematical difficulties. We assume this relaxation in our subsequent estimates. 
\end{remark}
In the course of proving this result, the authors show that there is a constant $\ctel{regell}$ depending only on $\lVert a \lVert_{L^{\infty}(0,T; C(\widebar \Omega))}$ and $\lVert \nu \lVert_{L^{2}(\Omega)}$ such that
\begin{equation}
\Vert p_{k}\Vert_{L^{\infty}(0,T; H^{1}(\Omega))} \leq \cter{regell}, \qquad \Vert \mathbf{u}_{k}\Vert_{L^{\infty}(0,T; L^{2}(\Omega)^{d})} \leq \cter{regell}. \label{eq:darcyest} 
\end{equation}
The next section details the additional estimates required to pass to the limit in \eqref{eq:fgconc}--\eqref{eq:darcyest} and hence prove Proposition \ref{prop:dreg}.


\subsection{Estimates}
\label{sec:regest}
We begin with some simple energy estimates, which in particular provide bounds on spatial derivatives of $c_{k}$.

\begin{proposition} \label{prop:energy}
Assume that \eqref{eq:fgconc}--\eqref{eq:fgell} hold. There exists $\ctel{energy}$ depending only on $\lVert a \lVert_{L^{\infty}(0,T; C(\widebar \Omega))}$ and $\lVert \nu \lVert_{L^{1}(\Omega)}$ such that for all $s \in [0, T]$,
\begin{multline}\label{eq:energy}
\phi_{\ast}\Vert c_{k}(s)\Vert_{L^{2}(\Omega)}^{2} + 2\int_{0}^{s}\int_{\Omega}\mathbf{D}_{k}(x, \mathbf{u}_{k}(x, t))\nabla c_{k}(x, t) \cdot \nabla c_{k}(x, t)\ud x \ud t  
\leq \cter{energy} 
\end{multline}
\end{proposition}

\begin{proof}
Take $\varphi = \mathbf{1}_{[0, s)}c_k$ as a test function in \eqref{eq:fgpara} to obtain
\begin{equation}
\begin{array}{l}
\dsp \int_{0}^{s}\langle \Phi\partial_{t}c_{k}(\cdot, t), c_{k}(\cdot, t)\rangle_{(H^{1})', H^{1}}\ud t \\
\dsp + \int_{0}^{s}\int_{\Omega} \mathbf{D}_{k}(x, \mathbf{u}_{k}(x, t))\nabla c_{k}(x, t) \cdot \nabla c_{k}(x, t)\ud x \ud t \\
\dsp - \int_{0}^{s}\int_{\Omega} c_{k}(x, t)\mathbf{u}_{k}(x, t)\cdot \nabla c_{k}(x, t) \ud x \ud t \\
\dsp + \int_{0}^{s}\int_{\Omega} c_{k}^{2}(x, t)b(x, t)\nu(x)\ud x \ud t \\
\dsp = \int_{0}^{s}\int_{\Omega} \hat c(x, t)c_{k}(x, t)a(x, t)\nu(x) \ud x \ud t.
\end{array} \label{eq:energy1}
\end{equation}
From \eqref{eq:fgconc}, $c_{k} \in L^{2}(0,T; H^{1}(\Omega))\cap L^{\infty}(\Omega \times (0, T))$. It follows that $\nabla \big(\tfrac{c_{k}^{2}}{2}\big) = c_{k}\nabla c_{k} \in L^{2}(0,T; L^{2}(\Omega))$ and so $\mathbf{1}_{[0,s)}\tfrac{c_{k}^{2}}{2}$ is a valid test function in \eqref{eq:fgpara}. We therefore rewrite the convection term
\begin{align*}
- \int_{0}^{s}\int_{\Omega} c_{k}(x, t)\mathbf{u}_{k}(x, t)&\lefteqn{\cdot \nabla c_{k}(x, t) \ud x \ud t }\\
&= - \int_{0}^{s} \int_{\Omega} \mathbf{u}_{k}(x, t)\cdot \nabla \Big(\frac{c_{k}^{2}(x, t)}{2}\Big)\ud x \ud t  \\
&= \int_{0}^{s} \int_{\Omega} \Big(\frac{c_{k}^{2}(x, t)}{2}\Big)(a(x, t) - b(x, t))\nu(x)\ud x \ud t.
\end{align*}
Integration by parts on the first term of \eqref{eq:energy1} yields 
\begin{equation*}
\int_{0}^{s}\langle \Phi\partial_{t}c_{k}(\cdot, t), c_{k}(\cdot, t)\rangle_{(H^{1})', H^{1}}\ud t = \frac{1}{2}\big(\lVert \sqrt{\Phi}c_{k}(\cdot, s)\lVert_{L^{2}(\Omega)}^{2} - \lVert \sqrt{\Phi}c_{0}\lVert_{L^{2}(\Omega)}^{2}\big).
\end{equation*}Substituting back in \eqref{eq:energy1},
\begin{eqnarray*}
\frac{1}{2}\big(\lVert \sqrt{\Phi}c_{k}(\cdot, s)\lVert_{L^{2}(\Omega)}^{2} \lefteqn{- \lVert \sqrt{\Phi}c_{0}\lVert_{L^{2}(\Omega)}^{2}\big)} &&\\
\qquad + \lefteqn{\int_{0}^{s}\int_{\Omega} \mathbf{D}_{k}(x, \mathbf{u}_{k}(x, t))\nabla c_{k}(x, t) \cdot \nabla c_{k}(x, t)\ud x \ud t }\\
&=& \int_{0}^{s}\int_{\Omega} \Big(\hat c(x, t)c_{k}(x,t) - \frac{c_{k}^{2}(x, t)}{2}\Big)a(x, t)\nu(x) \ud x \ud t \\
&&- \int_{0}^{s}\int_{\Omega} \frac{c_{k}^{2}(x,t)}{2}b(x,t)\nu(x)\ud x \ud t \\
&\leq& \big(\lVert a \lVert_{L^{\infty}(0,T; C(\widebar \Omega))} + \lVert b \lVert_{L^{\infty}(0,T; C(\widebar \Omega))}\big)\lVert \nu \lVert_{L^{1}(\Omega)}s ,
\end{eqnarray*}
where the last inequality follows from the uniform bound \eqref{eq:fgconc} and the hypotheses \eqref{hyp:porosity} and \eqref{hyp:initialconc} on the
porosity and the initial concentration. Estimate \eqref{eq:energy} is then a straightforward consequence.
\end{proof}

\begin{corollary} \label{cor:energy}
There exist $\ctel{dkhalf}$ and $\ctel{dk}$, both depending only on $\lVert a \lVert_{L^{\infty}(0,T; C(\widebar \Omega))}$ and $\lVert \nu \lVert_{L^{1}(\Omega)}$, such that
\begin{align}
\lVert \mathbf{D}_{k}^{1/2}(\cdot, \mathbf{u}_{k})\nabla c_{k}\lVert_{L^{2}(0,T; L^{2}(\Omega)^{d})} &\leq \cter{dkhalf} \label{eq:cor1}\\
\lVert \mathbf{D}_{k}(\cdot, \mathbf{u}_{k})\nabla c_{k}\lVert_{L^{2}(0,T; L^{4/3}(\Omega)^{d})} &\leq \cter{dk}. \label{eq:cor2}
\end{align}
\end{corollary}
\begin{proof}
Writing $\mathbf{D}_{k}(x, \mathbf{u}_{k})\nabla c_{k} \cdot \nabla c_{k} = 
|\mathbf{D}_{k}^{1/2}(x, \mathbf{u}_{k})\nabla c_{k}|^{2}$ and invoking
Proposition \ref{prop:energy} gives \eqref{eq:cor1}. One establishes that for a.e. $x \in \Omega$, 
\begin{equation}
|\mathbf{D}_{k}^{1/2}(x, \zeta)| \leq \Lambda_{\mathbf{D}}^{1/2}(1 + |\zeta|^{1/2}) \mbox{\quad for all $\zeta \in \RR^{d}$}. \label{eq:growthdsq}
\end{equation}
Hence, Estimate \eqref{eq:darcyest} shows that $\lVert \mathbf{D}^{1/2}_k(\cdot,\mathbf{u}_k) 
\lVert_{L^{\infty}(0,T; L^{4}(\Omega))}$ is bounded uniformly with respect to $k$.
Using H\"older's inequality to write
\begin{multline*}
\lVert \mathbf{D}_{k}(\cdot, \mathbf{u}_{k})\nabla c_{k}\lVert_{L^{2}(0,T; L^{4/3}(\Omega)^{d})} 
\\
\leq \lVert \mathbf{D}_{k}^{1/2}(\cdot, \mathbf{u}_{k})\lVert_{L^{\infty}(0,T; L^{4}(\Omega)^{d \times d})} \lVert \mathbf{D}_{k}^{1/2}(\cdot, \mathbf{u}_{k})\nabla c_{k}\lVert_{L^{2}(0,T; L^{2}(\Omega)^{d})}
\end{multline*}
thus gives \eqref{eq:cor2}.
\end{proof}

The problem of interest is dependent on both space and time and so it is standard
to complete these spatial estimates with time-derivative estimates on $c_k$, in
order to obtain the strong compactness of this sequence of functions.
\begin{proposition} \label{prop:dtc}
Let \eqref{eq:fgconc}--\eqref{eq:fgell} hold. Then there exists $\ctel{dtc1}$ depending only on $\lVert a \lVert_{L^{\infty}(0,T; C(\widebar \Omega))}$ and $\lVert \nu \lVert_{L^{2}(\Omega)}$ such that 
\begin{equation}
\lVert \Phi\partial_{t}c_{k}\lVert_{L^{2}(0,T; (W^{1, 4}(\Omega))')} \leq \cter{dtc1}. \label{eq:dtcest}
\end{equation}
\end{proposition}
\begin{proof}
Let $\varphi \in L^{2}(0,T; W^{1, 4}(\Omega))$. From \eqref{eq:fgconc} and \eqref{eq:fgpara} we have
\begin{multline*}
\Bigg|\int_{0}^{T}\langle \Phi\partial_{t}c_{k}(\cdot, t), \varphi(\cdot, t)\rangle_{(W^{1,4})', W^{1,4}}\ud t \Bigg| \\
\leq  \int_{0}^{T}\int_{\Omega} \big|\mathbf{D}_{k}(x, \mathbf{u}_{k}(x, t))\nabla c_{k}(x, t) \cdot \nabla \varphi(x, t)\big|\ud x \ud t \\
+ \int_{0}^{T}\int_{\Omega} \big|\mathbf{u}_{k}(x, t)\cdot \nabla \varphi(x, t)c_{k}(x, t)\big| \ud x \ud t\\+ \int_{0}^{T}\int_{\Omega} \big|c_{k}(x, t)\varphi(x, t)b(x, t)\nu(x)\big|\ud x \ud t \\
+ \int_{0}^{T}\int_{\Omega} \big|\hat c(x, t)\varphi(x, t)a(x, t)\nu(x) \big|\ud x \ud t \\
\leq \lVert \mathbf{D}_{k}(\cdot, \mathbf{u}_{k})\nabla c_{k}\lVert_{L^{2}(0,T; L^{4/3}(\Omega)^{d})} \lVert \nabla \varphi\lVert_{L^{2}(0,T; L^{4}(\Omega)^{d})} \\
+ \lVert \mathbf{u}_{k}\lVert_{L^{\infty}(0, T); L^{2}(\Omega)^{d})} \lVert \nabla \varphi\lVert_{L^{1}(0,T; L^{2}(\Omega)^{d})} \\
+ \lVert \varphi\lVert_{L^{1}(0,T; L^{2}(\Omega))}\lVert \nu\lVert_{L^{2}(\Omega)}\big( \lVert a\lVert_{L^{\infty}(0,T; C(\widebar \Omega)} + \lVert b\lVert_{L^{\infty}(0,T; C(\widebar \Omega)}\big).
\end{multline*}
We then use Estimates \eqref{eq:darcyest} and \eqref{eq:cor2} to obtain
\[
\left|\int_{0}^{T}\langle \Phi\partial_{t}c_{k}(\cdot, t), \varphi(\cdot, t)\rangle_{(W^{1,4})', W^{1,4}}\ud t \right| \le
\cter{dtc1}\lVert \varphi\lVert_{L^{2}(0,T; W^{1, 4}(\Omega))},
\]
which completes the proof. 
\end{proof}


\subsection{Passing to the limit} \label{sec:passlimittrunc}
In order to apply classical compactness results we introduce the Hilbert space
\[
\Phi H^{1}(\Omega) := \big\{ \Phi v \, | \, v \in H^{1}(\Omega) \big\}, 
\]
with norm
\begin{equation*}
\lVert w \lVert_{\Phi H^{1}(\Omega)} := \Big\lVert \frac{w}{\Phi} \Big\lVert_{H^{1}(\Omega)} \mbox{\quad for }w \in \Phi H^{1}(\Omega).
\end{equation*}
Observe that since the porosity is independent of time, the previous estimate \eqref{eq:dtcest} gives a bound in $L^{2}(0,T; (W^{1, 4}(\Omega))')$ of the sequence $(\partial_{t}(\Phi c_{k}))_{k \in \NN}$. Furthermore, applying \eqref{eq:dkassump} to the left-hand side of \eqref{eq:energy} shows that the sequence $(c_{k})_{k \in \NN}$ is bounded in $L^{2}(0,T; H^{1}(\Omega))$, and therefore that the sequence $(\Phi c_{k})_{k \in \NN}$ is bounded in $L^{2}(0,T; \Phi H^{1}(\Omega))$. One may verify that $\Phi H^{1}(\Omega)$ is compactly embedded in $L^{2}(\Omega)$, so that the invocation of a compactness result of Aubin \cite{au63} is valid. Combining these with the bound on the sequence $(c_{k})$ in $L^{\infty}(\Omega \times (0, T))$ from \eqref{eq:fgconc}, we have that, up to a subsequence,
\begin{align}
c_{k} &\to \widetilde{c} \mbox{\quad in $L^{\infty}(\Omega \times (0, T))$ weak-$\ast$}, \quad 0 \leq \widetilde{c} \leq 1 \mbox{ a.e. in }\Omega \times (0, T),\nonumber \\
c_{k} &\to \widetilde{c} \mbox{\quad weakly in $L^{2}(0,T; H^{1}(\Omega)$), proving \eqref{eq:regc}}, \label{eq:ckweakh1} \\
\Phi \partial_{t}c_{k} &\to \Phi \partial_{t}\widetilde{c} \mbox{\quad weakly in $L^{2}(0,T; (W^{1,4}(\Omega))')$, proving \eqref{eq:regdtc}},\nonumber \\
\Phi c_{k} &\to \Phi\widetilde{c} \mbox{\quad in $L^{2}(0,T; L^{2}(\Omega))$ and} \label{eq:ckaub1} \\
\Phi c_{k} &\to \Phi \widetilde{c} \mbox{\quad in $C([0, T]; (W^{1,4}(\Omega))'$, which proves \eqref{eq:regic})}. \nonumber 
\end{align}
From \eqref{eq:ckaub1} and the uniform positivity of $\Phi$ we deduce that, up to a subsequence, 
\begin{align*}
c_{k} &\to \widetilde{c} \mbox{\quad in $L^{2}(0,T; L^{2}(\Omega))$ and a.e. in $\Omega \times (0, T)$}. 
\end{align*}
Estimates \eqref{eq:darcyest} and the first step in the proof of \cite[Proposition 4.5]{fg00} demonstrates that there exists
$(\widetilde{p},\widetilde{\mathbf{u}})$ satisfying \eqref{eq:regpu} and \eqref{eq:regell} and that,
up to a subsequence,
\[
\nabla p_{k} \to \nabla \widetilde{p} \mbox{\quad in $L^{2}(0,T; L^{2}(\Omega)^{d})$} \mbox{\quad and} \quad
\widetilde{\mathbf{u}}_{k} \to \widetilde{\mathbf{u}} \mbox{\quad in $L^{2}(0,T; L^{2}(\Omega)^{d})$}.
\]
To prove Proposition \ref{prop:dreg}, it remains to establish \eqref{eq:dgradc}, and pass to the limit on \eqref{eq:fgpara} to demonstrate that $\widetilde{c}$ satisfies \eqref{eq:regpara}. To this end, we apply Corollary \ref{cor:h} to obtain $\mathbf{D}_{k}(\cdot, \mathbf{u}_{k}) \to \mathbf{D}(\cdot, \widetilde{\mathbf{u}})$ in $L^{2}(0,T; L^{2}(\Omega)^{d \times d})$. This, together with \eqref{eq:cor2} and \eqref{eq:ckweakh1} verifies the hypotheses of Lemma \ref{lem:ws} with $v_{k} = \nabla c_{k}$, $w_{k} = \mathbf{D}_{k}(\cdot, \mathbf{u}_{k})$, $a, r_{1}, r_{2}, s_{1}, s_{2}$ all 2 and $b=4/3$ and hence
\begin{equation*}
\mathbf{D}_{k}(\cdot, \mathbf{u}_{k})\nabla c_{k} \to \mathbf{D}(\cdot, \widetilde{\mathbf{u}})\nabla \widetilde{c} \mbox{\quad weakly in $L^{2}(0,T; L^{4/3}(\Omega)^{d})$,}
\end{equation*}
which proves \eqref{eq:dgradc}. Taking $\varphi \in L^{2}(0,T; W^{1, 4}(\Omega))$, we pass to the limit in \eqref{eq:fgpara} using weak-strong convergence. This concludes the proof of Proposition \ref{prop:dreg}.
 
\begin{remark}
Following Remark \ref{rem:rem1}, we show that $\mathbf{D}^{1/2}(\cdot, \widetilde{\mathbf{u}})\nabla \widetilde{c} \in L^{2}(0,T; L^{2}(\Omega)^{d})$ using the same technique as above. Apply Corollary \ref{cor:h} to obtain $\mathbf{D}_{k}^{1/2}(\cdot, \mathbf{u}_{k}) \to \mathbf{D}^{1/2}(\cdot, \widetilde{\mathbf{u}})$ in $L^{2}(0,T; L^{2}(\Omega)^{d \times d})$. As before, the uniform bound \eqref{eq:cor1} on the sequence $(\mathbf{D}_{k}^{1/2}(\cdot, \mathbf{u}_{k})\nabla c_{k})_{k \in \NN}$ and \eqref{eq:ckweakh1} allow us to apply Lemma \ref{lem:ws} and hence
\begin{equation}
\mathbf{D}_{k}^{1/2}(\cdot, \mathbf{u}_{k})\nabla c_{k} \to \mathbf{D}^{1/2}(\cdot, \widetilde{\mathbf{u}})\nabla \widetilde{c} \mbox{\quad weakly in $L^{2}(0,T; L^{2}(\Omega)^{d})$}. \label{eq:dsqgradcreg}
\end{equation}
We employ this result in the next section, where we establish the main result of this paper.
\end{remark}


\section{Existence of a solution with singular well data} \label{sec:measprob}

As stated in the introduction, we obtain the existence of a solution to \eqref{eq:darcy}--\eqref{eq:ic} with measure source terms by passing to the limit on the regularised problem.

\subsection{Stability result for the elliptic equation} \label{sec:stab}

The following proposition is a stronger version of \cite[Proposition 3.3]{fg00}.
In this reference, the stability result is proved under additional assumptions
on the measure $\mu$ and its approximation $(f_n)_{n\in\NN}$, stated in order to
get uniqueness of the limit via a regularity result of Meyers \cite{mey63}. We prove
here that, by invoking Stampacchia's notion of solution for linear elliptic equations
with measures \cite{dro00,pri95,stam65}, we can omit these additional
assumptions and prove nevertheless the convergence of the whole sequence of approximate solutions.

\begin{proposition}\label{prop:stab}
Let $(A_n)_{n\in\NN}\subset L^\infty(\O;M_d(\RR))$ be a sequence of bounded measurable matrix-valued functions,
$(F_n)_{n\in\NN}\subset L^2(\O)^d$, $(f_n)_{n\in\NN}\subset L^2(\O)$. We assume that,
for some $A:\O\to M_d(\RR)$, $F\in L^2(\O)^d$ and $\mu\in (C(\widebar \Omega))'$,
\begin{enumerate}
\item $A_n\to A$ pointwise a.e. on $\O$, $(A_n)_{n\in\NN}$ is bounded
in $L^\infty(\O;M_d(\RR))$ and uniformly elliptic,
\item $F_n\to F$ in $L^2(\O)^d$,
\item $f_n\to \mu$ for the weak-$*$ topology of $(C(\widebar{\O}))'$ and,
for all $n\in\NN$, $\int_\O f_n(x)\ud x=0$.
\end{enumerate}
We let $u_n$ be the weak solution to $-\div(A_n\nabla u_n-F_n)=f_n$ with homogeneous Neumann
boundary conditions. That is, 
\begin{equation}\label{pro:neu}
\left\{
\begin{array}{ll}
\dsp u_n\in H^1(\O)\mbox{ satisfies }\dsp\int_\O u_n(x)\ud x=0\mbox{ and, }
\forall v\in H^1(\O)\,,\\
\dsp\int_\O A_n(x)\nabla u_n(x)\cdot\nabla v(x)\ud x=\int_\O f_n(x)v(x)\ud x
 +\int_\O F_n(x)\cdot\nabla v(x)\ud x.
\end{array}
\right.
\end{equation}
Then there exists a weak solution $u$ to $-\div(A\nabla u-F)=\mu$ with
homogeneous Neumann boundary conditions:
\begin{equation}\label{pro:neumeas}
\left\{
\begin{array}{ll}
\dsp u\in \bigcap_{q<\frac{d}{d-1}}W^{1,q}(\O)\mbox{ satisfies }\dsp\int_\O u(x)\ud x=0\mbox{ and, }
\forall v\in \bigcup_{r>d}W^{1,r}(\O)\,,\\
\dsp\int_\O A(x)\nabla u(x)\cdot\nabla v(x)\ud x=\int_\O v(x)\ud\mu(x)
 +\int_\O F(x)\cdot\nabla v(x)\ud x,
\end{array}
\right.
\end{equation}
such that, as $n\to\infty$, $u_n\to u$ in $W^{1,q}(\O)$ for all $q<\frac{d}{d-1}$.
\end{proposition}

\begin{remark}
Note that the \emph{whole} sequence $u_n$ converges strongly to a solution of
\eqref{pro:neumeas}. The proof also shows that $(u_n)_{n\in\NN}$ actually converges to a solution
of $-\div(A\nabla u-F)=\mu$ in a stronger sense than \eqref{pro:neumeas},
namely Stampacchia's duality formulation, which ensures the uniqueness of the solution
to this problem (this is actually the reason why the stability result does not hold ``up to
a subsequence'').
\end{remark}

\begin{proof}
Let us first recall Stampacchia's setting for solving elliptic equations with measure
data. For a given bounded and coercive matrix-valued function $A$,
we define $T_A:(H^1(\O))'\to H^1(\O)$ as the inverse of $-\div(A^T\nabla )$ with
homogeneous Neuman boundary conditions, i.e. for $g\in (H^1(\O))'$, $w=T_A(g)$ is
the solution to
\begin{equation}\label{pro:neubase}
\left\{
\begin{array}{ll}
\dsp w\in H^1(\O)\,,\quad\int_\O w(x)\ud x=0\mbox{ and }\\
\dsp \forall v\in H^1(\O)\,,\;
\int_\O A(x)^T\nabla w(x)\cdot\nabla v(x)\ud x=\langle g,v\rangle_{(H^1)',H^1}.
\end{array}
\right.
\end{equation}
It has been proved (see e.g. \cite{stam65} for Dirichlet boundary conditions,
\cite{dro00} for other boundary conditions) that for any $p>d$,
$T^p_A=(T_A)_{|(W^{1,p'}(\O))'}$ takes values in $H^1(\O)\cap C(\widebar{\O})$.
Hence $(T^p_A)^*:\cM(\widebar{\O})+(H^1(\O))'\to W^{1,p'}(\O)$
is well defined. Stampacchia then defines, for $h\in 
\cM(\widebar{\O})+(H^1(\O))'$, the solution to
\begin{equation}\label{pro:neubasemeas}
\left\{
\begin{array}{ll}
-\div(A\nabla z)=h&\mbox{ in $\O$}\,,\\
A\nabla z\cdot\mathbf{n}=0&\mbox{ on $\partial\O$}
\end{array}
\right.
\end{equation}
as $z=(T^p_A)^*(h)$. It can easily be seen that $z$ does not depend on
$p$ (because $T^p_A=T^s_A$ on $(W^{1,\max(p',s')}(\O))'$),
and we therefore denote $z=(T_A)^*(h)$ without the $p$.

Several equivalent formulations exist that make clear why this $z$ can be considered
a re-formulation of \eqref{pro:neubasemeas}, see \cite{dro00,pri95,stam65}.
It can in particular be seen that if $h\in (H^1(\O))'$, then
$z=(T_A)^*(h)$ is the classical weak solution
of \eqref{pro:neubasemeas} with mean value zero.

In particular, defining $h_n\in (H^1(\O))'$ by $\langle h_n,\varphi\rangle_{(H^1)',H^1}=
\int_\O f_n(x)\varphi(x)\,dx + \int_\O F_n(x)\cdot \nabla\varphi(x)\,dx$,
we see that $u_n=(T_{A_n})^*(h_n)$. Notice that $h_n$ is the sum of two terms,
one converging to $\mu$ in $(C(\widebar{\O}))'$ weak-$*$ and the other
strongly converging to $\varphi\mapsto \int_\O F(x)\cdot\nabla \varphi(x)\,dx$
in $(H^1(\O))'$. We can therefore use the stability results 
of \cite{dro00,drophd} (established for Dirichlet, mixed or Robin boundary
conditions, but the proof is identical for Neumann boundary conditions)
which prove that, under the assumptions in the proposition,
$u_n\to u:=(T_A)^*(h)$ strongly in $W^{1,q}(\O)$ for all $q<\frac{d}{d-1}$,
where $h\in \cM(\widebar{\O})+(H^1(\O))'$ is defined
by 
\[
\forall \varphi \in C(\widebar{\O})\cap H^1(\O)\,,\quad
\langle h,\varphi \rangle_{\cM+(H^1)',C\cap H^1} = \int_\O \varphi(x)\ud\mu(x)+\int_\O F(x)\cdot\nabla \varphi(x)\ud x.
\]
The conclusion then follows immediately, as it is known \cite{dro00,pri95,stam65} that
$u=(T_A)^*(h)$ satisfies \eqref{pro:neumeas}. \end{proof}


\subsection{Approximation and estimates}

The first two steps of the proof of \cite[Theorem 2.1]{fg00} detail the approximations necessary to 
regularise the singular problem while retaining the no-flow compatibility conditions:
a sequence $(a_n,\nu_n)_{n\in\NN}$ is constructed such that
\bhyp{hyp:nunl1}
\nu_{n} \geq 0\,,\quad
\lVert \nu_{n} \lVert_{L^{1}(\Omega)} \leq \nu(\Omega),
\ehyp
\bhyp{hyp:anconverges}
a_{n}(x, t) \to a(x, t), \mbox{ as }n \to \infty, \mbox{ for a.e. }(x, t) \in \Omega \times (0, T),
\ehyp
\bhyp{hyp:anbound}
a_{n} \geq 0\mbox{ and }
(a_{n}) \mbox{ is bounded in }L^{\infty}(0,T; C(\widebar\Omega)), \mbox{ and }
\ehyp
\begin{equation}
\int_{\Omega}a_{n}(x, t)\nu_{n}(x)\ud x = \int_{\Omega}b(x, t)\nu_{n}(x)\ud x, \mbox{\quad for a.e. $t \in (0, T)$.} \label{hyp:ancompatible}
\end{equation}
This last property ensures that the compatibility condition \eqref{hyp:compatibility} still holds when $\nu$ is replaced by $\nu_{n}$ in \eqref{eq:darcy}--\eqref{eq:bc2}. We refer the reader to \cite{fg00} for further details.
\begin{remark}
The construction of $(\nu_{n})_{n \in \NN}$ necessary for the present work is simpler than in \cite{fg00}, since we do not require the sequence to be bounded in $(W^{1,q}(\Omega))'$ for all $q > 2$ (see Section \ref{sec:stab}).
\end{remark}
Replacing $a$ by $a_{n}$ and $\nu$ by $\nu_{n}$ in \eqref{eq:darcy}--\eqref{eq:ic}, Proposition \ref{prop:dreg} shows that there is a solution $(\widetilde{p}_{n}, \widetilde{\mathbf{u}}_{n}, \widetilde{c}_{n})$ to \eqref{eq:darcy}--\eqref{eq:ic} in the sense of \eqref{eq:regc}--\eqref{eq:regell}. Furthermore, \cite[Proposition 4.3]{fg00} shows that for any $1 \leq q < d/(d-1)$, there exists $\ctel{measell}$ not depending upon $n$ such that 
\begin{equation}
\Vert \widetilde{p}_{n}\Vert_{L^{\infty}(0,T; W^{1,q}(\Omega))} \leq \cter{measell}, \qquad \Vert \widetilde{\mathbf{u}}_{n}\Vert_{L^{\infty}(0,T; L^{q}(\Omega)^{d})} \leq \cter{measell}. \label{eq:measdarcyest} 
\end{equation}
It remains to re-examine the estimates of Section \ref{sec:regest} in light of \eqref{eq:measdarcyest}.

From Estimate \eqref{eq:energy} and Properties \eqref{hyp:nunl1}, \eqref{hyp:anbound}, one deduces that there exists $\ctel{measenergy}$ not depending
on $n$ such that
\begin{equation}
\lVert \widetilde{c}_{n}\lVert_{L^{2}(0,T; H^{1}(\Omega))} \leq \cter{measenergy}. \label{eq:measenergy}
\end{equation}
The elliptic estimates \eqref{eq:measdarcyest} yield an analogous result to Corollary \ref{cor:energy}.
\begin{corollary} \label{cor:measenergy}
Let \eqref{eq:regc}--\eqref{eq:regell} hold. Then 
\begin{itemize}
\item There exists $\ctel{meascor1}$, not depending on $n$, such that
\begin{align}
\lVert \mathbf{D}^{1/2}(\cdot, \widetilde{\mathbf{u}}_{n})\nabla \widetilde{c}_{n}\lVert_{L^{2}(0,T; L^{2}(\Omega)^{d})} &\leq \cter{meascor1} \label{eq:meascor1}.
\end{align}
\item For any $1 \leq r < 2d/(2d-1)$ there exists $\ctel{meascor2}$, not depending upon $n$, such that
\begin{align} 
\lVert \mathbf{D}(\cdot, \widetilde{\mathbf{u}}_{n})\nabla \widetilde{c}_{n} \lVert_{L^{2}(0,T; L^{r}(\Omega)^{d})} &\leq \cter{meascor2} \label{eq:meascor2}.
\end{align}
\end{itemize}
\end{corollary}
\begin{proof}
Estimate \eqref{eq:meascor1} follows from \eqref{eq:cor1}, \eqref{eq:dsqgradcreg}, \eqref{hyp:nunl1}
and \eqref{hyp:anbound}.

To prove \eqref{eq:meascor2}, note that for any $1 \leq q < d/(d-1)$, Estimate \eqref{eq:measdarcyest} gives $\ctel{meascor4}$ not depending upon $n$ such that
\begin{equation*}
\Vert |\widetilde{\mathbf{u}}_{n}|^{1/2} \Vert_{L^{\infty}(0,T; L^{2q}(\Omega))} \leq \cter{meascor4}.
\end{equation*}
It follows from the growth estimate \eqref{eq:growthdsq} that $\lVert \mathbf{D}^{1/2}(\cdot, \widetilde{\mathbf{u}}_{n})\lVert_{L^{\infty}(0,T; L^{2q}(\Omega)^{d \times d})}$ is bounded uniformly with respect to $n$. The H\"older inequality then gives
\begin{multline*}
\lVert \mathbf{D}(\cdot, \widetilde{\mathbf{u}}_{n})\nabla \widetilde{c}_{n}\lVert_{L^{2}(0,T; L^{r}(\Omega)^{d})} \\
\leq \lVert \mathbf{D}^{1/2}(\cdot, \widetilde{\mathbf{u}}_{n})\lVert_{L^{\infty}(0,T; L^{2q}(\Omega)^{d \times d})} \lVert \mathbf{D}^{1/2}(\cdot, \widetilde{\mathbf{u}}_{n})\nabla \widetilde{c}_{n}\lVert_{L^{2}(0,T; L^{2}(\Omega)^{d})},
\end{multline*}
where $r$ is chosen so that $\frac{1}{r} = \frac{1}{2q} + \frac{1}{2}$. Since $q$ can be any
number in $\big[1, \frac{d}{d-1}\big)$, $r$ can be arbitrarily chosen in $\big[ 1, \frac{2d}{2d-1}\big)$ and the proof is complete.
\end{proof}
Finally, analogous to Proposition \ref{prop:dtc}, the following estimate is necessary to utilise compactness results detailed in the next section. 

\begin{proposition}
Let \eqref{eq:regc}--\eqref{eq:regell} hold. Then for any $s > 2d$ there exists $\ctel{dtcmeas}$ not depending upon $n$ such that
\begin{equation}
\lVert \Phi \partial_{t}\widetilde{c}_{n}\lVert_{L^{2}(0,T; (W^{1, s}(\Omega))')} \leq \cter{dtcmeas}. \label{eq:measdtcest}
\end{equation}
\end{proposition}

\begin{proof}
Let $s>2d$ and $\varphi \in L^{2}(0,T; W^{1, s}(\Omega))$. From \eqref{eq:regpara},
\begin{multline*}
\Bigg|\int_{0}^{T}\langle \Phi \partial_{t}\tilde{c}_{n}(\cdot, t), \varphi(\cdot, t)\rangle_{(W^{1,s})', W^{1,s}}\ud t \Bigg|  \\
\leq \lVert \mathbf{D}(\cdot, \widetilde{\mathbf{u}}_{n})\nabla \widetilde{c}_{n}\lVert_{L^{2}(0,T; L^{r}(\Omega)^{d})}\lVert \nabla \varphi\lVert_{L^{2}(0,T; L^{r'}(\Omega)^{d})}  \\ 
+ \lVert \widetilde{\mathbf{u}}_{n}\lVert_{L^{\infty}(0, T); L^{r}(\Omega)^{d})}\lVert \nabla \varphi\lVert_{L^{1}(0,T; L^{r'}(\Omega)^{d})}  \\ 
+ \lVert \varphi\lVert_{L^{1}(0,T; L^{\infty}(\Omega))}\lVert \nu_{n}\lVert_{L^{1}(\Omega)}\big(\lVert a_{n}\lVert_{L^{\infty}(0,T; C(\widebar \Omega))} + \, \lVert b\lVert_{L^{\infty}(0,T; C(\widebar \Omega))}\big),
\end{multline*}
where $r$ is chosen such that $r'=s$, which imposes $r=s'<\frac{2d}{2d-1}<\frac{d}{d-1}$.
Using \eqref{hyp:nunl1}, \eqref{hyp:anbound}, \eqref{eq:measdarcyest}, \eqref{eq:meascor2} and
the embedding of $W^{1,s}(\O)$ into $L^\infty(\O)$, we infer that
\[
\left|\int_{0}^{T}\langle \Phi \partial_{t}\tilde{c}_{n}(\cdot, t), \varphi(\cdot, t)\rangle_{(W^{1,s})', W^{1,s}}\ud t \right|
\le \cter{cst:esttime} \lVert \varphi\lVert_{L^{2}(0,T; W^{1, s}(\Omega))}, 
\]
where $\ctel{cst:esttime}$ does not depend on $n$, hence the result.
\end{proof}


\subsection{Passing to the limit}

Analogously to Section \ref{sec:passlimittrunc}, from \eqref{eq:regc}, \eqref{eq:measenergy}, \eqref{eq:measdtcest} and once again using the Aubin compactness theorem, one may assume that, up to a subsequence, 
\begin{align}
\widetilde{c}_{n} &\to c \mbox{\quad weakly in $L^{2}(0,T; H^{1}(\Omega)$}), \label{eq:cntocweakh1} \\
\widetilde{c}_{n} &\to c \mbox{\quad in $L^{\infty}(\Omega \times (0, T))$ weak-$\ast$, \quad $0 \leq c \leq 1$ a.e. in $\Omega \times (0, T)$,}\label{eq:cntoclinfty} \\
\Phi \partial_{t}\widetilde{c}_{n} &\to \Phi\partial_{t}c \mbox{\quad weakly in $L^{2}(0,T; (W^{1,s}(\Omega))') \quad \forall s>2d$,}\label{eq:dtcntoc} \\
\Phi \widetilde{c}_{n} &\to \Phi c \mbox{\quad in $L^{2}(0,T; L^{2}(\Omega))$, and} \label{eq:cnaub1} \\
\Phi \widetilde{c}_{n} &\to \Phi c \mbox{\quad in $C([0, T]; (W^{1,s}(\Omega))') \quad \forall s > 2d$.} \label{eq:cnaub2}
\end{align}
From \eqref{eq:dtcntoc} we deduce \eqref{eq:measdtc}. Properties \eqref{eq:cnaub2} and
\eqref{eq:regic} give \eqref{eq:measic}. As before, from \eqref{eq:cnaub1} we infer that,
up to another subsequence,
\begin{equation}
\widetilde{c}_{n} \to c \mbox{\quad in $L^{2}(0,T; L^{2}(\Omega))$ and a.e. on $\Omega \times (0,T)$.}
\label{eq:cnae} 
\end{equation}
For the elliptic equation, from Estimates \eqref{eq:measdarcyest} we have, up to a subsequence,
\begin{align}
\widetilde{p}_{n} &\to p \mbox{\quad in $L^{\infty}(0,T; W^{1, q}(\Omega))$ weak-$\ast$ \quad $\forall 1 \leq q < d/(d-1)$, and} \label{eq:pntop} \\
\widetilde{\mathbf{u}}_{n} &\to \mathbf{u} \mbox{\quad in $L^{\infty}(0,T; L^{q}(\Omega)^{d})$ weak-$\ast$ \quad $\forall 1 \leq q < d/(d-1)$.} \label{eq:untou}
\end{align}
Taking $\psi \in L^{1}(0,T; W^{1, q}(\Omega))$ for some $q > d$, it is straightforward to pass to the limit on \eqref{eq:regell} using \eqref{eq:cnae}, \eqref{eq:pntop} and \eqref{eq:untou}. For the right-hand side of \eqref{eq:regell}, note that for $q > d$, $W^{1, q}(\Omega) \subset C(\widebar \Omega)$ and $(a_{n}(\cdot, t) - b(\cdot, t))\nu_{n} \to (a(\cdot, t) - b(\cdot, t))\nu$ in $(C(\widebar \Omega))'$ weak-$\ast$ (see below) and use the dominated convergence theorem in time. This proves \eqref{eq:measpu} and \eqref{eq:measell}.


We now invoke the stability result Proposition \ref{prop:stab} to obtain stronger convergence of the pressure and Darcy velocity, which is necessary to pass to the limit on the parabolic equation \eqref{eq:regpara}.

Let $t\in (0,T)$ such that $\tilde{c}_n(\cdot,t)\to c(\cdot,t)$ a.e. on $\O$ (almost
every $t$ satisfies this by \eqref{eq:cnae}). We set
\begin{align*}
A_{n}(x) &:= \frac{\mathbf{K}(x)}{\mu(\widetilde{c}_{n}(x,t))}, \quad F_{n}(x):= \rho(\widetilde{c}_{n}(x,t))\mathbf{g}, \\
A(x) &:= \frac{\mathbf{K}(x)}{\mu(c(x,t))}, \quad F(x):= \rho(c(x,t))\mathbf{g}.
\end{align*} 
Then from \eqref{hyp:K} and \eqref{hyp:densvisc}
it is clear that $A_n$ is uniformly bounded and elliptic.
By \eqref{eq:cnae} and the continuity of $\mu$ and $\rho$ on the unit interval, $A_{n}(\cdot) \to A(\cdot)$ a.e. on $\Omega$ and $F_{n}(\cdot) \to F(\cdot)$ in $L^{2}(\Omega)$ as $n \to \infty$,
so that the first two hypotheses of Proposition \ref{prop:stab} are satisfied. One also requires that
\begin{align*}
(a_{n}(\cdot, t) - b(\cdot, t))\nu_{n} \to (a(\cdot, t) - b(\cdot, t))\nu \mbox{ in $(C(\widebar \Omega))'$ weak-$\ast$,}
\end{align*}
which is established in \cite{fg00}. Then Proposition \ref{prop:stab} shows that $\widetilde{p}_n(\cdot,t)
\to p(\cdot,t)$ in $W^{1,q}(\Omega)$ for all $q<\frac{d}{d-1}$.
In particular, for a.e. $t \in (0, T)$, 
\begin{align}
\nabla \widetilde{p}_{n}(\cdot, t) \to \nabla p(\cdot, t)\mbox{ and } \widetilde{\mathbf{u}}_{n}(\cdot, t) \to \mathbf{u}(\cdot, t) \mbox{ in $L^{q}(\Omega)^{d}$ for all $1\leq q < \displaystyle \frac{d}{d-1}$,} \label{eq:strongpu}
\end{align}
without the extraction of subsequences after \eqref{eq:pntop}, \eqref{eq:untou}. The dominated convergence theorem therefore gives, using \eqref{eq:measdarcyest}, 
\begin{equation}
\displaystyle \widetilde{\mathbf{u}}_{n} \to \mathbf{u}
\mbox{\quad in $L^{p}(0,T; L^{q}(\Omega)^{d})$}\,,\;
 \forall p < \infty\,,\;\forall q < \frac{d}{d-1}. \label{eq:untoustrong}
\end{equation}
This strong convergence of the Darcy velocity is critical to addressing in particular the convergence of the diffusion-dispersion term. To this end, Corollary \ref{cor:h} with $H_{n} = \mathbf{D}^{1/2}$ gives
\begin{equation}
\mathbf{D}^{1/2}(\cdot, \widetilde{\mathbf{u}}_{n}) \to \mathbf{D}^{1/2}(\cdot, \mathbf{u}) \mbox{ in }L^{p}(0,T; L^{2q}(\Omega)^{d \times d})\,,\; \forall p < \infty, \forall q < \frac{d}{d-1}, \label{eq:dcntodc1}
\end{equation}
which, together with \eqref{eq:cntocweakh1} and Estimate \eqref{eq:meascor1}, allow us to apply Lemma \ref{lem:ws} with $v_{n} = \nabla \tilde{c}_{n}$, $w_{n} = \mathbf{D}^{1/2}(\cdot, \tilde{\mathbf{u}}_{n})$ and $a=b=r_{1}=r_{2}=s_{2}=2$, $s_{1}=2q \geq 2$. We obtain
\begin{equation}
\mathbf{D}^{1/2}(\cdot, \widetilde{\mathbf{u}}_{n})\nabla \widetilde{c}_{n} \to \mathbf{D}^{1/2}(\cdot, \mathbf{u})\nabla c \mbox{\quad weakly in $L^{2}(0,T; L^{2}(\Omega)^{d})$}. \label{eq:dcntodc2}
\end{equation}
Using \eqref{eq:dcntodc1}, \eqref{eq:dcntodc2} and Estimate \eqref{eq:meascor2} we apply Lemma \ref{lem:ws} again with $w_{n} = \mathbf{D}^{1/2}(\cdot, \widetilde{\mathbf{u}}_{n})$, $v_{n} = \mathbf{D}^{1/2}(\cdot, \widetilde{\mathbf{u}}_{n})\nabla \widetilde{c}_{n}$ and $a=r_{1}=r_{2}=s_{2}=2$, $s_{1}=2q\geq2$ and $b=r$ which yields
\begin{equation}
\mathbf{D}(\cdot, \widetilde{\mathbf{u}}_{n})\nabla \widetilde{c}_{n} \to \mathbf{D}(\cdot, \mathbf{u})\nabla c \mbox{ weakly in $L^{2}(0,T; L^{r}(\Omega)^{d})$}\,,\;
\forall r < \frac{2d}{2d-1}, \label{eq:dcntodc}
\end{equation}
and proves in particular \eqref{eq:measdgradc}. 

\begin{remark}
This splitting technique for obtaining convergence of the diffusion-dispersion term fills an apparent gap in the proof of \cite[Lemma 8]{ch04} or in the passing to the limit in \cite[Eq. (4.6)]{ch04}.
Using our notation, it seems to us that the weak convergence of  $\mathbf{D}(\cdot,\widetilde{\mathbf{u}}_n)\nabla\widetilde{c}_n$
and $\widetilde{\mathbf{u}}_n\cdot\nabla\widetilde{c}_n$
are justified in this reference by simply invoking 
the weak convergence of $\nabla \widetilde{c}_n$ in $L^2(\O\times (0,T))$ and the
strong convergence of $\widetilde{\mathbf{u}}_n$ in $L^q(\O\times(0,T))$ for $q<\frac{d+2}{d+1}$.
However, since $\frac{d+2}{d+1}<2$, a direct weak-strong convergence
argument cannot be applied and so more sophisticated reasoning is necessary.
The argument we propose above (the splitting of $\mathbf{D}(\cdot,\widetilde{\mathbf{u}}_n)$
into $\mathbf{D}(\cdot,\widetilde{\mathbf{u}}_n)^{1/2}\mathbf{D}(\cdot,\widetilde{\mathbf{u}}_n)^{1/2}$
and the double application of Lemma \ref{lem:ws})
can be adapted to the proof of the convergence of $\widetilde{\mathbf{u}}_n\cdot\nabla\widetilde{c}_n$
by writing
\[
\widetilde{\mathbf{u}}_n\cdot\nabla\widetilde{c}_n=(1+|\widetilde{\mathbf{u}}_n|^{1/2})
\frac{\widetilde{\mathbf{u}}_n}{1+|\widetilde{\mathbf{u}}_n|^{1/2}}\cdot\nabla\widetilde{c}_n
\]
and by using Lemma \ref{lem:ws} twice along with the fact that
$(\frac{\widetilde{\mathbf{u}}_n}{1+|\widetilde{\mathbf{u}}_n|^{1/2}}\cdot\nabla\widetilde{c}_n)_{n\in\NN}$ is bounded in
$L^2(\O\times(0,T))$.
\end{remark}

Suppose now that $\varphi \in L^{2}(0,T; W^{1, s}(\Omega)$ for some $s > 2d$. From \eqref{eq:regpara}, 
\begin{equation*}
\begin{array}{l}
\dsp \int_{0}^{T}\langle \Phi \partial_{t}\widetilde{c}(\cdot, t), \varphi(\cdot, t)\rangle_{(W^{1, 4})', W^{1, 4}}\ud t \\
\dsp + \int_{0}^{T}\int_{\Omega} \mathbf{D}(x, \widetilde{\mathbf{u}}(x, t))\nabla \widetilde{c}(x, t) \cdot \nabla \varphi(x, t)\ud x \ud t \\
\dsp - \int_{0}^{T}\int_{\Omega} \widetilde{\mathbf{u}}(x, t)\cdot \nabla \varphi(x, t)\widetilde{c}(x, t) \ud x \ud t \\
\dsp + \int_{0}^{T}\int_{\Omega} \widetilde{c}(x, t)\varphi(x, t)b(x, t)\nu(x)\ud x \ud t \\
\dsp =\int_{0}^{T}\int_{\Omega} \hat c(x, t)\varphi(x, t)a(x, t)\nu(x) \ud x \ud t ,
\end{array}
\end{equation*}
which we write as 
\begin{equation*}
T_{1} + T_{2} + T_{3} + T_{4} = T_{5}.
\end{equation*}
We pass to the limit on $T_{1}$ using \eqref{eq:dtcntoc} and on $T_{2}$ using \eqref{eq:dcntodc}. For $T_{3}$, we use \eqref{eq:cntoclinfty} and $\widetilde{\mathbf{u}}_{n} \cdot \nabla \varphi \to \mathbf{u}\cdot \nabla \varphi$ in $L^{1}(0,T; L^{1}(\Omega))$, which follows from \eqref{eq:untoustrong} and the choice of test function. We handle $T_{5}$ in a similar manner to the source terms in the elliptic equation using dominated convergence in time.

It remains to pass to the limit on $T_{4}$. Note that by the preceding arguments, $T_{4}$ has a limit as $n \to \infty$. To identify this limit, we refer the reader to \cite[Lemma 5.1]{fg00}, where the authors demonstrate that we can specify $c$ on null sets with respect to the Lebesgue measure in such a way that the last two lines of \eqref{eq:measc} are satisfied. This proves \eqref{eq:measpara} and thus concludes the proof of Theorem \ref{th:main}.
\appendix


\section{Some useful lemmas} \label{sec:app}

We present here the technical lemmas that we use to pass to the limit on the diffusion-dispersion term.

\begin{lemma} \label{lem:h}
Let $\Omega$ be a bounded subset of $\RR^{N}$, $N \in \NN$ and for each $n \in \NN$, let $H_{n}: \Omega \times \RR^{N} \to \RR$ be a Carath\'eodory function such that
\begin{itemize}
\item there exist positive constants $\ctel{hlem}$, $\gamma$ such that for a.e. $x \in \Omega$, 
\begin{equation}
|H_{n}(x, \xi)| \leq \cter{hlem}( 1 + |\xi|^{\gamma}), \mbox{\quad $\forall \xi \in \RR^{N}$, $\forall n \in \NN$,} \label{eq:lem1}
\end{equation}
\item there is a Carath\'eodory function $H: \Omega \times \RR^{N} \to \RR$ such that for a.e. $x \in \Omega$, 
\begin{equation}
H_{n}(x, \cdot) \to H(x, \cdot) \mbox{\quad uniformly on compact sets as $n \to \infty$.} \label{eq:lem2}
\end{equation} 
\end{itemize}
If $p \in [\gamma, \infty)$ and $(u_{n})_{n \in \NN} \subset L^{p}(\Omega)^{N}$ is a sequence with $u_{n} \to u$ in $L^{p}(\Omega)^{N}$ as $n \to \infty$, then
\begin{equation*}
H_{n}(\cdot, u_{n}) \to H(\cdot, u) \mbox{\quad in $L^{p/\gamma}(\Omega)$ as $n \to \infty$.}
\end{equation*}
\end{lemma}
\begin{proof}
First note that if $v \in L^{p}(\Omega)^{N}$, then $H_{n}(\cdot, v) \in L^{p/\gamma}(\Omega)$, since by \eqref{eq:lem1} there is $\ctel{hlem2}$ such that 
\begin{equation*}
\lVert H_{n}(\cdot, v)\lVert_{L^{p/\gamma}(\Omega)} \leq \cter{hlem2}(\mathrm{meas}(\Omega)^{\gamma/p} + \lVert v\lVert_{L^{p}(\Omega)^{N}}^{\gamma}).
\end{equation*}
Suppose that $H_{n}(\cdot, u_{n})$ does not converge to $H(\cdot, u)$ in $L^{p/\gamma}(\Omega)$. Then there is $\varepsilon > 0$ and a subsequence of $(H_{n}(\cdot, u_{n}))_{n \in \NN}$, still denoted by $(H_{n}(\cdot, u_{n}))_{n \in \NN}$, with 
\begin{equation}
\lVert H_{n}(\cdot, u_{n}) - H(\cdot, u)\lVert_{L^{p/\gamma}(\Omega)} \geq \varepsilon. \label{eq:contra} 
\end{equation}
Since $u_{n} \to u$ in $L^{p}(\Omega)^{N}$ as $n \to \infty$, the partial converse to the dominated convergence theorem establishes the existence of a subsequence of $(u_{n})_{n \in \NN}$, again still denoted by $(u_{n})_{n \in \NN}$, and a function $v \in L^{p}(\Omega)^{N}$ with $u_{n}(x) \to u(x)$ for a.e. $x \in \Omega$, and $|u_{n}(x)| \leq v(x)$ for a.e. $x \in \Omega$. Now, $H_{n}$ is Carath\'eodory, so together with \eqref{eq:lem2} it follows that $H_{n}(x, u_{n}(x)) \to H(x, u(x))$ for a.e. $x \in \Omega$. Furthermore, 
\begin{equation*}
|H_{n}(x, u_{n}(x))| \leq \cter{hlem}(1 + (v(x))^{\gamma}) \mbox{\quad for a.e. $x \in \Omega$},
\end{equation*}
so the dominated convergence theorem gives 
\begin{equation*}
H_{n}(\cdot, u_{n}) \to H(\cdot, u) \mbox{\quad in $L^{p/\gamma}(\Omega)$},
\end{equation*}
contrary to \eqref{eq:contra}.
\end{proof}
\begin{corollary} \label{cor:h}
Let $\Omega$, $H_n$ and $H$ satisfy the hypotheses of Lemma \ref{lem:h}. If $p, q \in [\gamma, \infty)$ and $(u_{n})_{n \in \NN} \subset L^{p}(0,T; L^{q}(\Omega)^{N})$ is a sequence with 
\[
u_n \to u \mbox{ in $L^{p}(0,T; L^{q}(\Omega)^{N})$ as $n \to \infty$, then}
\]
\begin{equation*}
H_{n}(\cdot, u_{n}) \to H(\cdot, u) \mbox{\quad in $L^{p/\gamma}(0,T; L^{q/\gamma}(\Omega))$ as $n \to \infty$.}
\end{equation*} 
\end{corollary}
\begin{proof}
Since $u_{n} \to u$ in $L^{p}(0,T; L^{q}(\Omega)^{N})$, the partial converse to the dominated convergence theorem in $L^{p}(0,T; L^{q}(\Omega)^{N})$ establishes the existence of a subsequence $(u_{n})_{n \in \NN}$, still denoted by $(u_{n})_{n \in \NN}$, and a function $v \in L^{p}(0, T)$ such that for a.e. $t \in (0, T)$,
\begin{align*} 
u_{n}(\cdot, t) \to u(\cdot, t) \mbox{\quad in $L^{q}(\Omega)^{N}$, and }\lVert u_{n}(\cdot, t) \lVert_{L^{q}(\Omega)^{N}} \leq v(t).
\end{align*}
Now, by Lemma \ref{lem:h}, for a.e. $t \in (0, T)$,
\begin{equation*}
H_{n}(\cdot, u_{n}(\cdot, t)) \to H(\cdot, u(\cdot, t)) \mbox{\quad in $L^{q/\gamma}(\Omega)$.}
\end{equation*}
Furthermore,
\begin{equation*}
\lVert H_{n}(\cdot, u_{n}(\cdot, t))\lVert_{L^{q/\gamma}(\Omega)} \leq \cter{hlem}(1 + (v(t))^{\gamma}) \mbox{\quad for a.e. $t \in (0, T)$},
\end{equation*}
so the dominated convergence theorem gives
\begin{equation*}
H_{n}(\cdot, u_{n}) \to H(\cdot, u) \mbox{\quad in $L^{p/\gamma}(0,T; L^{q/\gamma}(\Omega))$},
\end{equation*}
which concludes the proof. 
\end{proof}
\begin{lemma} \label{lem:ws}
Let $\Omega$ be a bounded, open subset of $\RR^{N}$ and for each $n \in \NN$, let $w_{n}: \Omega \times (0, T) \to \RR$ and $v_{n}: \Omega \times (0, T) \to \RR$ be such that 
\begin{align*}
w_{n} &\to w \mbox{\quad in $L^{r_{1}}(0,T; L^{s_{1}}(\Omega))$, and}\\
v_{n} &\to v \mbox{\quad weakly in $L^{r_{2}}(0,T; L^{s_{2}}(\Omega))$,}
\end{align*}
where $r_{1}, r_{2}, s_{1}, s_{2} \geq 1$ are such that $1/r_{1} + 1/r_{2} \leq 1$ and $1/s_{1} + 1/s_{2} \leq 1$. Suppose also that the sequence $(w_{n}v_{n})_{n \in \NN}$ is bounded in $L^{a}(0,T; L^{b}(\Omega))$, where $a,b \in (1, \infty)$. Then
\begin{equation*}
w_{n}v_{n} \to wv \mbox{\quad weakly in $L^{a}(0,T; L^{b}(\Omega))$.}
\end{equation*}
\end{lemma}
\begin{proof}
The assumption that $(w_{n}v_{n})_{n \in \NN}$ is bounded in $L^{a}(0,T; L^{b}(\Omega))$ yields the existence of $\chi \in L^{a}(0,T; L^{b}(\Omega))$ such that, up to a subsequence, 
\begin{equation*}
w_{n}v_{n} \to \chi \mbox{\quad weakly in $L^{a}(0,T; L^{b}(\Omega))$},
\end{equation*}
which implies that $w_{n}v_{n} \to \chi$ in the sense of distributions on $\Omega \times (0, T)$. From the other hypotheses, $L^{r_{1}}(0,T; L^{s_{1}}(\Omega)) \hookrightarrow L^{r_{2}}(0,T; L^{s_{2}}(\Omega))'$, so using weak-strong convergence we have $w_{n}v_{n} \to wv$ in the sense of distributions on $\Omega \times (0, T)$, hence $\chi = wv$. 
\end{proof}

\bibliography{peaceman-measure.bib}
\bibliographystyle{siam}

\end{document}